\documentclass[psamsfonts]{amsart}

\usepackage{amssymb,amsfonts}
\usepackage[all,arc]{xy}
\usepackage{enumerate}
\usepackage{mathrsfs}
\usepackage{graphicx}

\newtheorem{thm}{Theorem}[section]

\newtheorem{prop}[thm]{Proposition}
\newtheorem{lem}[thm]{Lemma}

\theoremstyle{definition}

\theoremstyle{remark}
\newtheorem{rem}[thm]{Remark}

\newcommand{\op}{\mathop{\mathrm{op}}}

\newcommand{\B}{\mathop{\mathcal{B}_0}}
\newcommand{\Z}{\mathop{\mathbb{Z}}}

\makeatletter
\let\c@equation\c@thm
\makeatother
\numberwithin{equation}{section}

\title{A Generating Function for the Distribution of Runs in Binary Words }

\author{James J. Madden}

\date{July 10, 2017}

\begin{document}

\begin{abstract}
Let $N(n,r,k)$ denote the number of binary words of length $n$ that begin with $0$ and contain exactly $k$ runs (i.e., maximal subwords of identical consecutive symbols)  of length $r$.  We show that the generating function for the sequence $N(n,r,0)$, $n=0,1,\ldots$, is $(1-x)(1-2x + x^r-x^{r+1})^{-1}$ and that the generating function for  $\{N(n,r,k)\}$ is $x^{kr}$ time the $k+1$ power of this.   We extend to counts of words containing exactly $k$ runs of $1$s by using symmetries on the set of binary words.

\medskip
\noindent \textsc{Key words.} runs, maximal runs, distribution of runs, generating function, Binary word, Bernouilli trials

\medskip
\noindent  AMS Subject Classification:  Primary 05A15.  Secondary 60C05

\end{abstract}

\maketitle

\section{Statement of the main results}

By a maximal run in a binary word, we mean a maximal consecutive sub-sequence of identical symbols, cf.\ \cite{feller}, page 42.  We will use the words ``run'' and ``maximal run'' synonymously.  A run at the beginning of a word is an initial sequence of identical symbols that is followed immediately by a different symbol, and a run at the end is defined analogously.    If we cut a word between every pair of consecutive symbols that differ, the subwords that result are the runs in that word.  

The purpose of this note is to present generating functions that count the number of binary words of length $n$ ($n=0,1,2,\ldots$) that have a specified number of runs of a specified length.  Our main result is as follows:

\begin{thm} Let $N(n,r,k)$ be the number of binary words of length $n$ that begin with 0 and contain exactly $k$ maximal runs of length $r$. Then 
$$\sum_{n=0}^{\infty}N(n,r,k)x^n=x^{k\,r}\left(\frac{1 - x}{1 - 2 x + x^r-x^{r+1}}\right)^{k+1}.$$\end{thm} 

This result is quite elementary, but we have not been able to find it in any other source.  The sequences with $r=1$ and $k<5$, and some sequences with $r=2$ appear in the \textit{On-Line Encyclopedia of Integer Sequences}, but the close relationships between them is not made clear in the comments there.  

The proof of the theorem is given in sections 3 and 4.  We count sequences that begin with $0$ because this results in significant technical simplifications.  Obviously, we can deduce the count of all sequences with a specified number of runs of a specified length using symmetry. 

Some researchers are more interested in counting \textit{success runs} than in counting all runs.  By a success run, we mean a maximal consecutive sub-sequence of $1$s.  Let $M(n,r,k)$ be the number of binary words of length $n$ that begin with 0 and contain exactly $k$ maximal runs of $1$s of length $r$.

\begin{prop}  $M(n,r,k) = N(n, r+1,k)$ for all $n$, $r$ and $k$.\end{prop}

The proof is given in section 5.

For a binary word $b$ of length $n$, define $K_n^r(b)$ to be the number runs of length exactly $r$ in $b$.  If we view $K_n^r$ as a random variable for a binomial experiment with parameters $(n,1/2)$, then $P(K_n^r = k) = N(n,r,k)/2^{n-1}$.  Our results permit the rapid computation of the exact distribution of $K_n^r$ for $n$ up to a thousand and any $r$ less than $n$.  We include some graphs and computing times in section $6$.

Now, let $M_n^{(r)}(b)$ be the number of success runs in $b$ that have length at least $r$.  Museli \cite{mus} has given an elementary formula for the probability  $P(M_n^{(r)}=m)$, where $M_n^{(r)}$ is viewed as a random variable on the sample space of an $(n,p)$ binomial  experiment.  The proposition shows how Museli's results are related to ours.  

Sinha \& Sinha \cite{sinha} used a generating function to attempt to derive a formula for $M(n,r,k)$, which in their notation is $N_n^{k,r}$. However, the generating function that they use is different from ours and is used to count different objects, and they require additional combinatorial arguments to get the formula they seek.  We have checked the formula they give against known counts, and we do not find agreement, so we have either misunderstood their notation or there is an error in their formula.

\section{Preliminaries}

Let $\B(n,r,k)$ denote the set of binary words $b$ that satisfy the following conditions:
\begin{itemize}
\item $b$ is of length $n$;
\item if $n>0$, $b$ begins with $0$;
\item $b$ contains exactly $k$ maximal runs of length $r$.
\end{itemize} 
For small $n$, we have:
\begin{align*}
\B(n,r,k)&=\emptyset,\;\hbox{if $n<0$, since there are no words of negative length;}\\
\B(0,r,k) &=\{\ast\},\;\hbox{where $\ast$ denotes the empty word;}\\
\B(1,1,0)&=\emptyset\;\hbox{and $\B(1,1,1) = \{0\}$};\\
\B(1,r,0)&=\{0\},\;\hbox{if $r=2,3,\ldots$.}
\end{align*}
For example:
\begin{align*}
\B(6,1,0)&=\{ 000000, 000011, 000111, 001100, 001111\};\\
\B(6,2,2)&=\{001101, 001001, 001011, 011001, 011011, 010011\}.
\end{align*}

\section{Proof of theorem in case $k=0$}

We define the numbers $W(n,r)$, $n\in \Z$, $r=1,2,\ldots$, as the coefficients of the power series expansion of the rational function in the equation below.  We write $W_r(x)$ for the power series. 

\begin{equation}\label{eq1}W_r(x): = \sum_{n=0}^{\infty}W(n,r)x^n=\frac{1 - x}{1 - 2 x + x^r-x^{r+1}}.\end{equation}
For small $n$, we have:
\begin{align*}
W(n,r) &= 0,\;\mbox{if } n < 0;\\
W(0,r) &= 1;\\
W(1,1) &= 0;\\ 
W(1,r) &=  1,\;\hbox{if $r=2,3,\ldots$}.
\end{align*}
Moreover, for all $r=1,2,\ldots$, $W(n,r)$ is defined by the following recursion:
$$W(n,r) = 2\,W(n - 1,r) - W(n - r,r) + W(n - r - 1,r).$$

\begin{rem} $W(n,1)$ is the Fibonacci sequence with a 1 prepended, which is A212804 in the \textit{On-Line Encyclopedia of Integer Sequences}.  $W(n,2)$ is an offset of A005251;  $W(n,3)$ is an offset of A049856; $W(n,4)$ is an offset of A108758.  The sequences $W(n, r)$ with $r\geq5$  do not appear in the \textit{OEIS} at the present time.\end{rem}

\begin{lem} $W(n,r) = N(n,r,0)$ for all integers $n$ and for all positive integers $r$.\end{lem}

\begin{proof} To simplify notation, we set $N(n,r):=N(n,r,0)$.  We must show that $W(n,r) = N(n,r)$.  We can see that $N(n,r)$ satisfies the initial conditions for $W(n,r)$ by inspection of the data already given.  To verify the recursion formula, we consider cases:

Case 1: $1<n<r$.  In this case, no binary word of length $n$ contains a run of length $r$.  Therefore, $N(n,r) = 2^{n-1} = 2\,N(n-1,r)$, since we are counting words that begin with $0$.   Since $n<r$,  $W(n - r,r)=0=W(n - r - 1,r)$. Thus $N(n,r)$ satisfies the recursive rule of $W(n,r)$ for $n=2,3,\ldots r-1$. 

Case 2: $1<n=r$.  We must show that $N(r,r) = 2\,N(r-1,r) -1$.  But there is only one binary word beginning with $0$ and having length $r$ that is not in $\B(r,r,0)$, namely the word of $r$ $0$s.  So this case is clear.  In particular, $N(r,r) = 2^{r-1}-1$.  

Case 3: $n=r+1$.  We must show that $N(r+1,r) = 2\,N(r,r) - N(1,r) + 1$.  In the special case $r=1$, we need to show  $N(2,1) = 2\,N(1,1) - N(1,1) + 1 = 0+0+1$.  This is evident, since the only word of length 2 beginning with 0 and having no runs of length 1 is $00$.  If $r>1$, we need to show $N(r+1,r) = 2\,N(r,r)$.  We have already verified that $N(r,r) = 2^{r-1}-1$.  There are only two binary words beginning with 0 and having length $r+1$ that contain runs of length $r$, namely $01\cdots1$ and $0\cdots 01$, and hence $N(r+1,r) = 2^r-2$, which is what we sought to show.

Case 4: $n>r+1$. (This case contains the key idea in this note.)  We must show that $$N(n,r) = 2\,N(n - 1,r) - N(n - r,r) + N(n - r - 1,r).$$  There are $N(n - 1,r)$ binary words of length $n-1$ with no runs of length $r$.  We create $2\,N(n - 1,r)$ binary words by writing either a 0 or a 1 at the end of each of these, but in doing so, we may create a run of length $r$.  There are $N(n - r,r)$ words of length $n-1$ that end with a run of length $r-1$.  (These arise from the elements $b$ of $\B(n-r,r,0)$ by appending to $b$ a run of length $r-1$ symbols that are different from the last symbol in $b$.)  Each of these becomes a word containing a run of length $r$ in one way by the addition of a symbol at the end.  Thus, when we extend words from $\B(n-1,r,0)$ by appending a symbol, we create $2\,N(n - 1,r) - N(n - r, r)$ words of length $n$ with no runs of length $r$.  

Some words of length $n-1$ contain a run of length $r$ at the end, and no other runs of length $r$.  There are exactly $N(n - r - 1,r)$ of these, since we make each one by taking an $r$-run-free word of length $n - r - 1$ (i.e., and element of $\B(n-r-1,r, 0)$) and appending a run of $r$ symbols different from the last symbol of the taken word.  Upon the addition of one more copy of the same symbol, we create a word that is free of runs of length $r$, thus adding $N(n - r - 1,r)$ elements to the set we have formed.  

We have now described the only ways we can create or destroy a run of length $r$ by appending a symbol.  This shows that the recursive formula for $N(n,r)$ is valid and completes the proof of the $k=0$ case.   \end{proof}

\section{Proof of theorem for $k>0$}

If $w$ is a binary word, then $\op(w)$ denotes the word obtained from $w$ by writing $1$ in place of $0$ and $0$ in place of $1$.

Let $w\in \B(n,r,k)$.  Then $w$ has the structure:
 $$w=w_0r_1w_1r_2\cdots w_{k-1}r_kw_k,$$  
 where $r_i$ is a run of $0$s or of $1$s of length $r$ and  $w_i\in\B(n_i,r,0)$.  Note that $w_i$ may be empty (i.e., $n_i=0$).  Also note that $n_0+\cdots+n_k = n-kr.$ 

 Given a word structure as above, suppose that for each $i=0,\ldots, k$ we select a word $b_i$ in $\B(n_i,r,0)$.  Then (we claim), there is only one way to substitute runs of $0$s or $1$s for the $r_i$ and either $b_i$ or $\op(b_i)$ for $w_i$ ($i=0,\ldots,k$) in order to make a word in $\B(n,r,k)$.  We can see this inductively as follows.  If $w_0$ is empty, then $r_1$ must consist of $0$s.  Otherwise, in place of $w_0$, we write $b_0$.  Now, suppose we have completed filling in values for the $r_i$ and $w_i$ up to a given point.  The last symbol in the word formed thus far either belongs to a run, or to some run-free $w_i$.  If we are now to add a run, we must use symbols other than the last one appearing.  If we are to add a run-free segment (in place some non-empty $w_j$) then we must add either $b_j$ or $\op(b_j)$ and we have only one choice, lest we extend the last run, $r_j$.

The argument in the last paragraph shows that $$N(n,r,k) = \sum_{n_1,\cdots,n_k}\left\{\,\prod_{i=1}^kN(n_i,r)\,\Big|\, n_0+\cdots+n_k=n-rk\,\right\}.$$  But this is the coefficient of $x^{n-rk}$ in $(W_r(x))^{k+1}$, where $W_r(x)$ is the power series in \ref{eq1}.  The equation in the theorem follows immediately.

\section{Proof of the  Proposition}

Let $\B(n)$ denote the set of all binary words of length $n$ that begin with $0$.  We define a bijection
$\gamma:\B(n)\to \B(n)$ that codes runs of length $r$ as subwords  of the form $01\cdots1$ of length $r$.  Specifically, suppose $b = r_0r_1\cdots r_k$ is an element of $\B(n)$ written as a concatenation of runs $r_i$.  Necessarily, $r_0$ is a run of $0$s, $r_1$ is a run of $1$s and so on.   In general, the runs with even index consist of $0$s, while those of odd index consist of $1$s.  To compute $\gamma(b)$, we proceed as follows. In place of each $r_i$, write a word of the same length as $r_i$ consisting of a $0$ followed by $1$s.  The inverse of $\gamma$ is straightforward to construct.  Again, suppose  $b\in\B(n)$.  Then $b$ may be written in the form $s_0s_1\cdots s_k$, where each $s_i$ is a $0$ followed by zero or more consecutive $1$s.  We construct $\gamma^{-1}(b)$ as follows. In place of  $s_0$, write a string of $0$s of length equal to the length of $s_0$; in place of $s_1$ write a string of $1$s of length equal to the length of $s_1$, and so on.  In general, we write strings of $0$s in place of the $s_i$ when $i$ is even and strings of $1$s in place of $s_i$ when $i$ is odd, always writing a string of length equal to the one we are replacing.  This is obviously the inverse of $\gamma$, so we see that $\gamma$ is a bijection. Now, it is clearly the case that $\gamma(\B(n,r+1,k))$ consists of exactly those strings in $\B(n)$ that contain $k$ runs of $1$s of length $r$.  This proves the proposition.

\section{Computations}

The figure below shows the probability mass function  of $K^1_{240}$ to the right, in red.   To make the data more visible, we have filled in the  region bounded above by the polygonal path joining $N(240,1,k)/2^{239}$, for $k=30, 31,\ldots,100$.    We also show the PMF of  $K^2_{240}$ in the middle, in orange (for $k=10, 11,\ldots, 50$), and the PMF of $K^3_{240}$ to the left, in green (for  $k=10, 11,\ldots, 50$).  Using the function $\mathtt{SeriesCoefficient}$  in \textit{Mathematica} on a 2013 MacBook Pro, it takes about $0.75$ seconds to compute the 138 data points shown here.  It takes about 2 minutes to compute the whole list of 1001 numbers  $N(1000,1,k)/2^{999}$, $0\leq k\leq 1000$.

\bigskip

\begin{center}
\includegraphics[width=.66\textwidth]{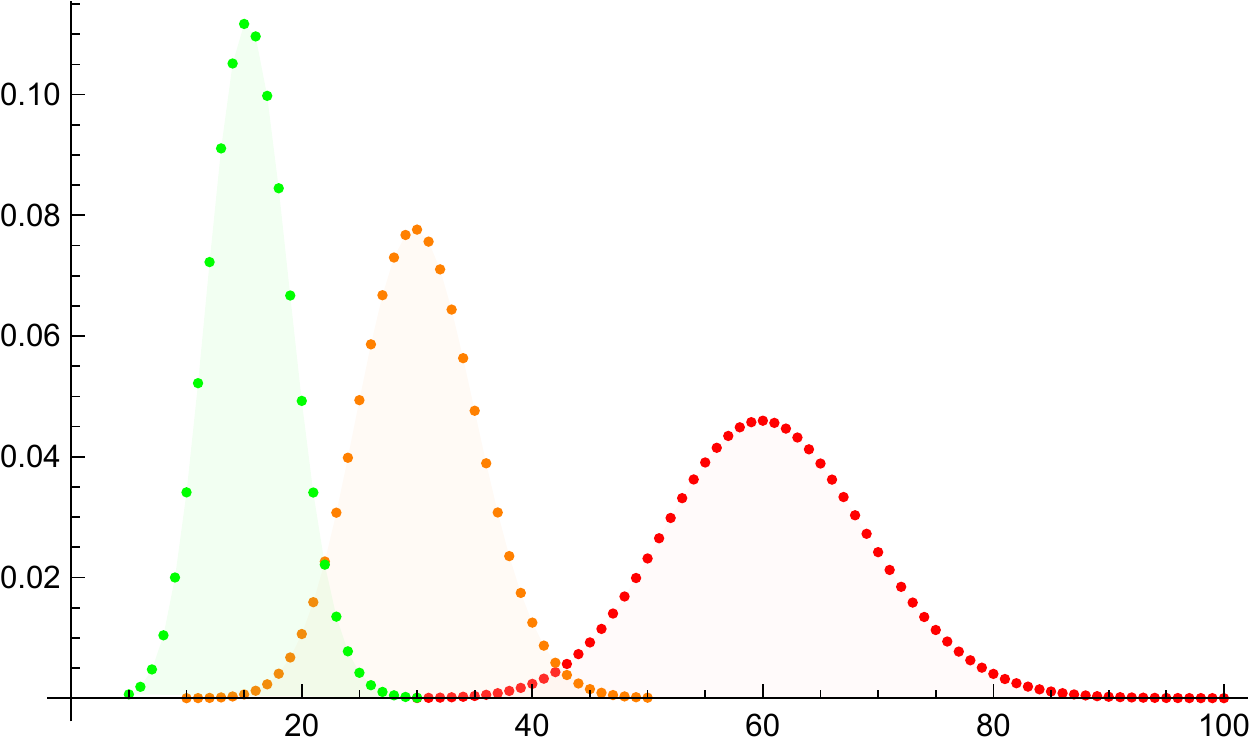}
\end{center}

\bigskip

James J. Madden

Department of Mathematics

222 Prescott Hall

Louisiana State University

Baton Rouge LA 70803-4918

madden@math.lsu.edu

jamesjmadden@gmail.com

\end{document}